\documentclass[leqno,a4paper,11pt]{article}
\usepackage{amsmath,amssymb,ascmac,latexsym,amsthm}
\usepackage[dvips]{graphicx}
\theoremstyle{definition}
\newtheorem{Assu}{Assumption}

\newtheorem{Th}{Theorem}
\newtheorem{Def}{Definition}
\newtheorem{rem}{Remark}
\newtheorem{Lem}{Lemma}
\newtheorem{Prop}[Lem]{Proposition}

\title{Characterization of the ranges of wave operators for Schr\"odinger equations with time-dependent short-range potentials via wave packet transform}
\author{Taisuke Yoneyama and Keiichi Kato}
\date{\today}

\begin{document}
\maketitle
  \begin{abstract}
In this paper, we give a characterization of the ranges of the wave operators 
for Schr\"odinger equations with time-dependent ``short-range'' potentials by using wave packet transform,
which is different from the one in Kitada--Yajima \cite{KY}.
We also give an alternative proof of the existence of the wave operators for time-dependent potentials, which has been firstly proved by D. R. Yafaev \cite{Yafaev}.
\end{abstract}
%
%
\section{Introduction}

In this paper,
we consider the following Schr\"odinger equation with time-dependent short-range potentials:
\begin{align}
\label{siki-1-nonini}
i\frac{\partial}{\partial t}u = H(t)u, \quad H(t)= H_0+V(t),\quad H_0
=-\frac12\sum_{j=1}^{n}\displaystyle\frac{\partial^2}{\partial x_j^2}=-\frac12\Delta
\end{align}
in the Hilbert space $\mathcal{H}=L^2(\mathbb{R}^n)$, where
$V(t)$ is the multiplication operator of a function $V(t,x)$
and the domain $D(H_0)=H^2(\mathbb{R}^{n})$ is the Sobolev space of order two.
We give a characterization of the ranges of the wave operators
for Schr\"odinger equation with time-dependent potentials which are short-range in space by using wave packet transform,
which is different from the characterization in Kitada--Yajima \cite{KY}.
We also give an alternative proof of the existence of the wave operators,
which has been firstly proved by D. R. Yafaev \cite{Yafaev}.

The wave packet transform is defined by A. C\'ordoba and C. Fefferman \cite{CF}.
The second author, M. Kobayashi and S. Ito have introduced the physically natural representation of solutions to Schr\"odinger equations via
wave packet transform with time dependent wave packet (\cite{KMS}, \cite{KKS2}), by which we characterize the ranges of the wave operators.
If the potential $V(t,x)$ does not depend on time, our characterization space coincides with the continuous spectral subspace.

In order to prove the existence of the wave operators,
we use Cook--Kuroda's method (\cite{cook}, \cite{kuroda}) with the representation via the wave packet transform.

The wave operators for Schr\"odinger equations have been studied from 1950s.
For time-independent short-range potentials, J. Cook \cite{cook} and S. T. Kuroda \cite{kuroda} have shown that the wave operators exist and that
their ranges are the absolutely continuous spectral subspace.
E. Mourre \cite{Mourre} has proved the absence of singular continuous spectrum.
V. Enss \cite{Enss} has given an alternative proof of the existence and the completeness (coincidence of the range and the continuous spectral subspace)
of the wave operators,
which is called time-dependent scattering theory.
For time-dependent short-range potentials,
D. R. Yafaev \cite{Yafaev} has shown the existence of the wave operators.
H. Kitada--K. Yajima \cite{KY} has proved the existence of the wave operators and the modified wave operators for time-dependent short-range potentials and for time-dependent long-range potentials, respectively,
and has characterized their ranges.

We assume that $V(t,x)$ satisfies the following conditions, which is called short-range.
\begin{Assu}
(i) $V(t,x)$ is a real-valued Lebesgue measurable function of $(t,x)\in\mathbb{R}\times\mathbb{R}^{n}$.\\
\hspace{9pt}(ii) There exist real constants $\delta>1$ and $C>0$ such that
\begin{align*}
|V(t,x)|\leq 
C(1+|x|)^{-\delta}
\end{align*}
for $(t,x)\in\mathbb{R}\times\mathbb{R}^{n}$.
\end{Assu}
\begin{Assu}
There exists a family of unitary operators $(U(t,\tau))_{(t,\tau)\in\mathbb{R}^2}$ in $\mathcal{H}$  satisfying the following conditions.

\noindent\hspace{11pt}(i) For $f\in \mathcal{H}$, $U(t,\tau)f$ is strongly continuous function with respect to $t$ and satisfies
\[
U(t,\tau')U(\tau',\tau)=U(t,\tau),\,U(t,t)=I\quad\mbox{for all }t,\tau',\tau\in\mathbb{R},
\]
where $I$ is the identity operator on $\mathcal{H}$.\\
\hspace{9pt}(ii) For $f\in H^2(\mathbb{R}^{n})$, $U(t,\tau)f$ is strongly continuously differentiable in $\mathcal{H}$ with respect to $t$ and satisfies
\[
\frac{\partial}{\partial t}U(t,\tau)f=-iH(t)U(t,\tau)f\quad\mbox{for all }t,\tau\in\mathbb{R}.
\]
\end{Assu}
\begin{rem}
If Assumption $\mathrm{(A)}$ is satisfied and $V(t)f$ is strongly differentiable in $\mathcal{H}$ for $f\in H^2(\mathbb{R}^{n})$, Assumption $\mathrm{(B)}$ is satisfied (c.f. T. Kato \cite{TKato}).
\end{rem}
Let $\mathcal{S}$ be the Schwartz space of all rapidly decreasing functions on $\mathbb{R}^n$ and
$\mathcal{S}'$ be the space of tempered distributions on $\mathbb{R}^n$.
For positive constants $a$ and $R$, we put
$\Gamma_{a,R}=\{(x,\xi)\in\mathbb{R}^n\times\mathbb{R}^n\Big|\,|\xi|\leq a$ or $|x|\geq R\}$
and 
$
\mathcal{S}_{scat}=\left\{\Phi\in\mathcal{S}\Big|\,\|\Phi\|_{\mathcal{H}}=1 \mbox{ and }\hat{\Phi}(0)\neq0\right\}.
$
\begin{Def}[Wave packet transform]\label{wpt-def}
Let $\varphi\in\mathcal{S}\setminus\left\{ 0\right\}$ and $f\in\mathcal{S'}$.
We define the wave packet transform $W_\varphi f(x,\xi)$ of $f$ 
with the wave packet generated by a function $\varphi$ as follows:
\[
W_\varphi f(x,\xi)=\int_{\mathbb{R}^{n}}\overline{\varphi(y-x)}f(y)e^{-iy\xi}dy
\quad\mbox{for } (x,\xi)\in\mathbb{R}^{n}\times\mathbb{R}^{n}.
\]
Its inverse is the operator $W_{\varphi}^{-1}$ which is defined by
\[
W_\varphi^{-1} F(x)=\frac{1}{(2\pi)^n\|\varphi\|_{\mathcal{H}}^2}\int\int_{\mathbb{R}^{2n}}\varphi(x-y)F(y,\xi)e^{ix\xi}dyd\xi
\]
for $x\in\mathbb{R}^{n}$ and a function $F(x,\xi)$ on $\mathbb{R}^{n}\times\mathbb{R}^{n}$.
\end{Def}
\begin{Def}\label{scat-def}
Let $\Phi\in\mathcal{S}_{scat}$ and we put $\Phi(t)=e^{-itH_0}\Phi$.
We define $\tilde{D}_{scat}^{\pm,\Phi}(\tau)$ by the set of all functions in $\mathcal{H}$ such that
\begin{align*}
\lim_{t\rightarrow\pm\infty}\left\|\chi_{\Gamma_{a,R}}W_{\Phi(t-\tau)}[U(t,\tau)f](x+(t-\tau)\xi,\xi)\right\|_{L^2(\mathbb{R}^{n}_x\times\mathbb{R}^{n}_\xi)}=0
\end{align*}
for some positive constants $a$ and $R$.
For $\tau\in\mathbb{R}$, $D_{scat}^{\pm,\Phi}(\tau)$ is defined by the closure of $\tilde{D}_{scat}^{\pm,\Phi}(\tau)$ in the topology of $\mathcal{H}$.
\end{Def}

The main state of this paper is the following.
\begin{Th}\label{main-theorem}
Suppose that $\mathrm{(A)}$ and $\mathrm{(B)}$ be satisfied.
Then the wave operators
\begin{align*}
W_{\pm}(\tau)=\mbox{s-}\hspace{-2mm}\lim_{t\rightarrow\pm\infty}U(\tau,t)e^{-i(t-\tau)H_0}
\end{align*}
exist for any $\tau\in\mathbb{R}$ 
and their ranges $\mathcal{R}(W_{\pm}(\tau))$ coincide with $D_{scat}^{\pm,\Phi}(\tau)$ for any $\Phi\in\mathcal{S}_{scat}$.
In particular, $D_{scat}^{\pm,\Phi}(\tau)$ is independent of $\Phi$.
\end{Th}

We use the following notations throughout the paper.
$i=\sqrt{-1},\,n\in\mathbb{N}$. 
For a subset $\Omega$ in $\mathbb{R}^{n}$ or in $\mathbb{R}^{2n}$, the inner product
and the norm on $L^2(\Omega)$ are defined by $(f,g)_{L^2(\Omega)}=\int_\Omega f\bar{g}dx$
and $\|f\|_{L^2(\Omega)}=(f,f)_{L^2(\Omega)}^{1/2}$ for $f,g\in L^2(\Omega)$, respectively.
We write $\partial_{x_j}=\partial/\partial x_j$, $\partial_t=\partial/\partial t$,
$L^2_{x,\xi}=L^2(\mathbb{R}^{n}_x\times\mathbb{R}^{n}_\xi)$,
$(\cdot,\cdot)=(\cdot,\cdot)_{L^2_{x,\xi}}$, $\|\cdot\|=\|\cdot\|_{L^2_{x,\xi}}$,
$\langle t\rangle=1+|t|$, $\|f\|_{\Sigma(l)}=\sum_{|\alpha+\beta|=l}\|x^\beta\partial^\alpha_x f\|_{\mathcal{H}}$
and $W_\varphi u(t,x,\xi)=W_\varphi[u(t)](x,\xi)$.
$\mathcal{F}$ and $\mathcal{F}^{-1}$ are the Fourier transform and the inverse Fourier transform defined by
$
\mathcal{F}f(\xi)=\hat{f}(\xi)=\int_{\mathbb{R}^{n}} e^{-ix\cdot\xi}f(x)dx
$
and $\mathcal{F}^{-1}f(\xi)=(2\pi)^{-n}\int_{\mathbb{R}^{n}} e^{ix\cdot\xi}f(\xi)d\xi$, respectively.
We often write $\{\xi=0\}$ as $\{(x,\xi)\in\mathbb{R}^{2n}|\,\xi=0\}$.
For sets $A$ and $B$, $A\setminus B$ denotes the set $\{a\in A|\,a\notin B\}$.
$\chi_A(x)$ the characterization function of a measurable set $A$, which is defined by $\chi_A(x)=1$ on $A$ and $\chi_A(x)=0$ otherwise.
$F(\cdots)$ denotes the multiplication operator of a function $\chi_{\{x\in\mathbb{R}^n|\cdots\}}(x)$.
For an operator $T$ on $\mathcal{H}$, $D(T)$ and $\mathcal{R}(T)$ denote the domain and the range of $T$, respectively.
$\mathcal{H}_{p}(T)$ and $\mathcal{H}_{p}(T)^{\perp}$ denote pure point subspace of a self-adjoint operator $T$ on $\mathcal{H}$ and its orthogonal complement space, respectively.

In the preceding studies (\cite{Enss}, \cite{Yafaev}, \cite{KY}), the proofs of the existence of the wave operators are relied on Cook -- Kuroda's method.
In our proof, we use the duality argument and the representation
\begin{align}
\label{o-m}
W_{\varphi_ 0}[e^{itH_0}&U(t,0)\psi](x,\xi)\\
&=W_{\varphi_0}\psi(x,\xi)-i\int_0^t e^{i\frac{1}{2}s |\xi|^2}W_{\varphi(s)}[V(s)U(s,0)\psi](x+s\xi,\xi)ds,\nonumber
\end{align} 
which is developed by the second author, M. Kobayashi and S. Ito (\cite{KMS}, \cite{KKS2}).
Thus our proof can be regarded as a variation of Cook--Kuroda's method by using the wave packet transform.

V. Enss \cite{Enss} and H. Kitada--K. Yajima \cite{KY} use the phase space decomposition operators $P_{\pm,R}$ and $P_{0,R}$ as follows:
$
P_{\pm,R}f(x)=\int\int\int_{|z|\geq R} e^{i(x-y)\xi}g^a_{\pm}(z,\xi)\eta(y-z)f(y)dzdyd\xi$
and
$P_{0,R}f(x)=\int\int\int_{|z|< R} e^{i(x-y)\xi}\eta(y-z)f(y)dzdyd\xi.\nonumber
$
Here $g^a_+(z,\xi)$ ($g^a_-(z,\xi)$) is smooth cut-off function whose support is contained in the set that $|\xi|\leq a$ 
and $z \cdot \xi <0 (>0)$
and $\eta$ is a smooth function such that  $\int\eta dx=1$ and  $\mathrm{supp}\,\hat{\eta}$ is included in a small ball in $\mathbb{R}^n$.
Since
$P_{-,R}\,U(t,0)f$, 
$P_{0,R}\,U(t,0)f\to0$ as $t\to+\infty$ and
$P_{+,R}\,U(t,0)f\sim P_{+,R}\,e^{-itH_0}f$ as $R\to\infty$,
we have $U(t,0)f\sim P_{+,R}\,e^{-itH_0}f$ as $R,t\to\infty$.
By using this formula, 
V. Enss \cite{Enss} has proved that the ranges of the wave operators are the continuous spectral subspace of time-independent Hamiltonian $H=H_0+V$
and H. Kitada--K. Yajima \cite{KY} has characterized the ranges of the wave operators for time-independent potential.

On the contrary, our proof is simple. We decompose the phase space $\mathbb{R}^n_x\times \mathbb{R}^n_\xi$ into only two parts 
$\Gamma_{a,R}$ and $\Gamma_{a,R}^c$, and estimate the wave packet transform of the solution in each part.

In the case that $V$ does not depend on $t$,
the following well-known theorem holds for $H=H_0+V$.
We give an alternative proof of the theorem by using our characterization.
\begin{Th}[J. Cook \cite{cook}, S. T. Kuroda \cite{kuroda}, E. Mourre \cite{Mourre} and V. Enss \cite{Enss}]\label{indp-theorem}
Suppose that $\mathrm{(A)}$ be satisfied and that $V$ do not depend on $t$.
Then the wave operators $W_\pm=\mbox{s-}\hspace{-0mm}\lim_{t\rightarrow\pm\infty}e^{i(t-\tau)H}e^{-i(t-\tau)H_0}$
exist, are independent of $\tau$ and are strongly complete:
\begin{align}
\label{indp-comp}
\mathcal{R}(W_{\pm})=\mathcal{H}_{p}^\perp(H).
\end{align}
\end{Th}

\begin{rem}\label{KYtype-def}
We can give \textit{''Kitada--Yajima'' type} characterization spaces via the wave packet transform.
Let $\Phi\in\mathcal{S}_{scat}$ and $\mathcal{K}_{a,N}=\{(x,\xi)\in\mathbb{R}^n\times\mathbb{R}^n\Big\vert|\,\xi|\leq a\mbox{ or }|x|\leq N\}$.
We define $\mathcal{Y}_{scat}^{\pm}(\tau)$ by the closure of the set of all functions in $\mathcal{H}$ such that
\begin{align}
\label{KYtype-scat}
\lim_{N\to\pm\infty}\left\|\chi_{\mathcal{K}_{a,N}}W_{\Phi}[U(t_N^{\pm},\tau)f]\right\|=0
\end{align}
for some positive constant $a$ and some sequences $(t_N^{\pm})$ tending to $\pm\infty$ as $N\to\infty$.
\end{rem}
\noindent We have $\mathcal{R}(W_{\pm}(\tau))=\mathcal{Y}_{scat}^{\pm}$, which is remarked in Section 5.

\begin{rem}\label{rem-ana}
Consider the case that $n\geq2$.
Let $a$ be a non-negative constant, $\sigma\in(0,1]$ and $\Phi\in\mathcal{S}_{scat}$.
We put 
$$\tilde{\Gamma}_{a,\sigma}^\pm=\left\{(x,\xi)\in\mathbb{R}^n\times\mathbb{R}^n\Big\vert\,|\xi|\leq a\mbox{ or }\mp\cos\theta(x,\xi)\geq\sigma\right\},$$
where $\cos\theta(x,\xi)=(x\cdot\xi)/|x||\xi|$ and double-sign corresponds.
Let $A_{scat}^{\pm}(\tau)$ be the closure of the set of all functions in $\mathcal{H}$ satisfying
that
\begin{align*}
\lim_{t\to\pm\infty}\|\chi_{\tilde{\Gamma}_{a,\sigma}^\pm}W_{\Phi(t)}[U(t,\tau)f](x+(t-\tau)\xi,\xi)\|=0
\end{align*}
for some positive constants $a$ and $\sigma\in (0,1)$.
Then we get
\begin{align*}
A_{scat}^{\pm}(\tau)=\mathcal{R}(W_\pm(\tau))=D_{scat}^{\pm}(\tau),
\end{align*}
since $W_{\Phi}^{-1}(C_0^\infty(\mathbb{R}^{2n}\setminus\tilde{\Gamma}_{0,1}^+))$ and 
$W_{\Phi}^{-1}(C_0^\infty(\mathbb{R}^{2n}\setminus\tilde{\Gamma}_{0,1}^-))$ are dense in $\mathcal{H}$.
\end{rem}
The plan of this paper is as follows.
In section 2, we recall the properties of the wave packet transform.
In section 3, we give a proof of the existence of the wave operators (the former part of Theorem \ref{main-theorem}).
In section 4, we give a proof of the characterization of the ranges of the wave operators (the latter part of Theorem \ref{main-theorem}).
In section 5, we prove Theorem \ref{indp-theorem} by using our characterization.
%
%
\section{Wave packet transform}
In this section, we briefly recall the properties of the wave packet transform
and give the representation of solutions to $(\ref{siki-1-nonini})$ via wave packet transform,
which is introduced in \cite{KMS}, \cite{KKS2}.
\begin{Prop}\label{wpt-prop}
Let $\varphi\in\mathcal{S}\setminus\left\{ 0\right\}$ and $f\in\mathcal{S'}$.
Then the wave packet transform $W_\varphi f(x,\xi)$ has 
the following properties:\\
$\mathrm{(i)}$ $W_\varphi f(x,\xi)\in C^\infty(\mathbb{R}_x^{n}\times\mathbb{R}_\xi^{n})$.\\
$\mathrm{(ii)}$ If $f$ and $\varphi$ depend on $t$ and satisfy that $f\in C^1(\mathbb{R};\mathcal{S}')$ and $\varphi\in C^1(\mathbb{R};\mathcal{S})$, we have
\begin{align*}
W_{\varphi(t)}[\partial_xf](t,x,\xi)
=-i\xi W_{\varphi(t)}f(t,x,\xi)+W_{\partial_x\varphi(t)}f(t,x,\xi) \mbox{ in }\mathcal{S'}
\end{align*}
and
\begin{align*}
W_{\varphi(t)}[\partial_tf](t,x,\xi)
=\partial_tW_{\varphi(t)}f(t,x,\xi)-W_{\partial_t\varphi(t)}f(t,x,\xi) \mbox{ in }\mathcal{S'}.
\end{align*}
$\mathrm{(iii)}$ If $f,g\in \mathcal{H}$ and $\psi\in\mathcal{S}\setminus\left\{ 0\right\}$, we have
\begin{align}
\label{siki-wpt-inv}
(W_\varphi f,W_\psi g)
&=\overline{(\varphi,\psi )_{\mathcal{H}}}( f,g )_{\mathcal{H}}=(\psi,\varphi)_{\mathcal{H}}( f,g )_{\mathcal{H}}.
\end{align}
$\mathrm{(iv)}$ The inversion formula
$
W_\varphi^{-1}[W_\varphi f]=f
$
holds for $f\in\mathcal{S'}$
\end{Prop}
We can easily show Proposition \ref{wpt-prop}.

Let $\varphi_0\in\mathcal{S}\setminus\left\{ 0\right\}$ and $\psi\in \mathcal{H}$.
We consider the following initial value problem of $(\ref{siki-1-nonini})$:
\begin{align}
\label{siki-1}
&\begin{cases}
i\partial_tu+\frac{1}{2}\Delta u -V(t,x)u = 0, \quad (t,x)\in\mathbb{R}\times\mathbb{R}^{n},\\
u(t_0)=\psi,\hspace{90pt} x\in\mathbb{R}^n.
\end{cases}\end{align}
Let $\varphi(t,x)=e^{-itH_0}\varphi_0(x)$.
We have from Proposition \ref{wpt-prop}
\begin{align*}
W_{\varphi(t)}[\Delta u](t,x,\xi)
=W_{\Delta\varphi(t)}u(t,x,\xi)+2i\xi\cdot\nabla_xW_{\varphi(t)}u(t,x,\xi)
-|\xi|^2W_{\varphi(t)}u(t,x,\xi).
\end{align*}
$(\ref{siki-1})$ is transformed to
\begin{align*}
&\begin{cases}
\Big(i\partial_t+i\xi\cdot\nabla_x-\frac{1}{2}|\xi|^2\Big)W_{\varphi(t)}u(t,x,\xi) = W_{\varphi(t)}\left[V(t)u(t)\right](x,\xi),\\
W_{\varphi(t_0)}u(t_0,x,\xi)=W_{\varphi_0}\psi(x,\xi),
\end{cases}
\end{align*}
since $W_{\varphi(t)}[i\partial_tu](t,x,\xi)=i\partial_tW_{\varphi(t)}u(t,x,\xi)
+W_{i\partial_t\varphi(t)}u(t,x,\xi)$.
We have by the method of characteristic curve that
\begin{align}
\label{siki-psi-V}
W_{\varphi(t)}[U(t,t_0)\psi](x,\xi)=&e^{-i\frac{1}{2}(t-t_0) |\xi|^2}
W_{\varphi_0}\psi(x-(t-t_0)\xi,\xi)\\
&-i\int_{t_0}^t e^{-i\frac{1}{2}(t-s) |\xi|^2}
W_{\varphi(s)}\left[V(s)U(s,t_0)\psi\right](x-(t-s)\xi,\xi)ds.\nonumber
\end{align}
In particular, we have for the case that $V\equiv0$
\begin{align}
\label{siki-psi}
W_{\varphi(t)}&[e^{-itH_0}\psi](x,\xi)=e^{-i\frac{1}{2}t|\xi|^2}W_{\varphi_0}\psi(x-t\xi,\xi).
\end{align}
Combining $(\ref{siki-psi-V})$ and $(\ref{siki-psi})$,
we have $(\ref{o-m})$
and
\begin{align}
\label{siki-psi2}
&W_{\varphi(t)}[U(t,t')e^{-it'H_0}\psi](x+t\xi,\xi)\\
&=W_{\varphi_0}\psi(x,\xi)
+i\int_t^{t'} e^{-i\frac{1}{2}(t-s) |\xi|^2}
W_{\varphi(s)}\left[V(s)U(s,t')e^{-it'H_0}\psi\right](x+s\xi,\xi)ds.\nonumber
\end{align}

%
%
\section{Existence of the wave operators}
In this section, we prove the existence of the wave operators by using the wave packet transform which is defined in the previous section.

The following well-known lemma is used in the proof of Lemma \ref{no2-term}.
\begin{Lem}\label{kuroda-lem}
Let $f\in\mathcal{S}$.
Suppose that $\mathrm{supp}\,\hat{f}\subset K$ with some compact set $K$ which does not contain the origin.
For any open set $K'\supset K$ and any non-negative integer $l$, there exists a positive constant $C$ such that
\begin{align*}
|e^{-itH_0}f(x)|\leq C\langle x\rangle^{-l}\|f\|_{\Sigma(l)}
\end{align*}
for any $(t,x)\in\mathbb{R}\times\mathbb{R}^{n}$ with $x/t\notin K'$ and $t\neq 0$.
\end{Lem}
\begin{proof}
See \cite{kuroda2} or \cite{Hor}.
\end{proof}
Using the above lemma, we obtain the following lemma.
\begin{Lem}\label{no2-term}
Suppose that $\mathrm{(A)}$ be satisfied.
Let $a$ and $R$ be positive constants.
Then for any $L\in(0,a/6]$ and $\varphi_0\in\mathcal{S}\setminus\{0\}$ with $\mathrm{supp}\,\hat{\varphi}_0\subset\{\xi\in\mathbb{R}^n|\,L/2<|\xi|<L\}$,
there exists a positive constant $C$ satisfying
\begin{align}
\label{siki-Prop-no2}
\big\|W_{\varphi(s)}\left[V(s)\psi\right](x+s\xi,\xi)\big\|_{L^2(\mathbb{R}^{2n}\setminus\Gamma_{a,R})}\leq C\langle s\rangle^{-\delta}\|\psi\|_{\mathcal{H}}
\end{align}
for any $s\geq0$ and any $\psi\in \mathcal{H}$.
\end{Lem}
\begin{proof}
Let $\rho=a/6$ and let $l$ be an integer satisfying $l\geq\delta+(n+1)/2$.
We put 
$
V_\rho(t,x)=\chi_0(\frac{1}{\rho \langle t\rangle}x)V(t,x),
$
where $\chi_0\in C^\infty(\mathbb{R}^n)$ satisfies $\chi_0(x)=1$ for $|x|\geq1$ and $\chi_0(x)=0$ for $|x|\leq1/2$.
Thus there exists a positive constant $C$ such that
$|V_\rho(t,x)|\leq C\langle t\rangle^{-\delta}$ for
any $t\in\mathbb{R}$ and any $x\in\mathbb{R}^n$.

If $(x,\xi)\in\mathbb{R}^{2n}\setminus\Gamma_{a,R}$, $s\geq\max\{3R/a,3\}$ and $y\in\mathbb{R}^{n}$ satisfying $|y-(x+s\xi)|\leq as/3$,
then we have
\begin{align}
\label{v-2}
|y|\geq|x+s\xi|-|y-(x+s\xi)|\geq (as-R)-\frac{as}{3}\geq\rho\langle s\rangle.
\end{align}

In this proof, we write $\{\cdots\}$ as $\{y\in\mathbb{R}^n|\cdots\}$.
Since 
$
V_\rho(s,y)=V(s,y)
$
for $|y|\geq\rho \langle s\rangle$,
we have by Plancherel's theorem and $(\ref{v-2})$
\begin{align*}
&\left\|\int_{\{|y-(x+s\xi)|\leq as/3\}}\overline{e^{-isH_0}\varphi_0(y-(x+s\xi))}V(s,y)\psi(y)e^{-i\xi y}dy\right\|_{L^2(\mathbb{R}^{2n}\setminus\Gamma_{a,R})}\\
\leq&\left\|\int_{\{|y-x|\leq as/3\}}\overline{e^{-isH_0}\varphi_0(y-x)}V_{\rho}(s,y)\psi(y)e^{-i\xi y}dy\right\|\\
\leq&\left\|\overline{e^{-isH_0}\varphi_0(y-x)}V_{\rho}(s,y)\psi(y)\right\|_{L^2(\mathbb{R}_x^n\times\mathbb{R}_y^n)}\\
\leq& C\langle s\rangle^{-\delta}\|\psi\|_{\mathcal{H}}.
\end{align*}

On the other hand,
Lemma $\ref{kuroda-lem}$ shows that
\begin{align*}
&\left\|\int_{\{|y-(x+s\xi)|> \frac{as}{3}\}}\overline{e^{-isH_0}\varphi_0(y-(x+s\xi))}V(s,y)\psi(y)e^{-i\xi y}dy\right\|\nonumber\\
=&\left\|\left(\chi_{\{|y-x|> \frac{as}{3}\}}\overline{e^{-isH_0}\varphi_0(y-x)}\right)V(s,y)\psi(y)\right\|_{L^2(\mathbb{R}_x^n\times\mathbb{R}_y^n)}\nonumber\\
\leq& C\langle s\rangle^{-l+(n+1)/2}\|\varphi_{0}\|_{\Sigma(l)}\left\|\langle y-x\rangle^{-(n+1)/2}V(s,y)\psi(y)\right\|_{L^2(\mathbb{R}_x^n\times\mathbb{R}_y^n)}\nonumber\\
=&C\langle s\rangle^{-l+(n+1)/2}\|\varphi_{0}\|_{\Sigma(l)}\left\|\langle x\rangle^{-(n+1)/2}\right\|_{\mathcal{H}}\left\|\psi\right\|_{\mathcal{H}}\nonumber\\
\leq& C\langle s\rangle^{-\delta}\|\psi\|_{\mathcal{H}}\nonumber,
\end{align*}
since $\mathrm{supp}\,\hat{\varphi}_0\subset\{\xi\in\mathbb{R}^n|\,0<|\xi|<a/6\}$.

Hence $(\ref{siki-Prop-no2})$ is obtained.
\end{proof}

We shall use the following lemma in Section 5.
Let $\Gamma_a^{b,\pm}=\{(x,\xi)\in\mathbb{R}^{2n}|\,|\xi|\geq a, |x|\geq b \mbox{ and }\pm x\cdot\xi\geq0\}$.
\begin{Lem}\label{no2-term2}
Suppose that $\mathrm{(A)}$ be satisfied.
Let $a$ and $b$ be positive constants.
Then for any $L\in(0,a/6]$ and $\varphi_0\in\mathcal{S}\setminus\{0\}$ with $\mathrm{supp}\,\hat{\varphi}_0\subset\{\xi\in\mathbb{R}^n|\,L/2<|\xi|<L\}$,
there exists a positive constant $C$ independent of $b$ satisfying
\begin{align*}
\big\|W_{\varphi(s)}\left[V(s)\psi\right](x\pm s\xi,\xi)\big\|_{L^2(\Gamma_a^{b,\pm})}\leq C\langle s+b\rangle^{-\delta}\|\psi\|_{\mathcal{H}}
\end{align*}
for any $s\geq0$ and any $\psi\in \mathcal{H}$.
\end{Lem}
\begin{proof}
The proof is obtained by the same argument as in the proof of Lemma \ref{no2-term} and the estimate that
\begin{align*}
|y|&\geq|x\pm s\xi|-|y-(x\pm s\xi)|\\
&\geq \sqrt{|x|^2\pm2x\cdot\xi+|\xi|^2}-\frac13a(s+b)\\
&\geq \frac1{\sqrt2}a(s+b)-\frac13a(s+b)\geq\frac{1}{6}a\langle s+b\rangle,
\end{align*}
if $(x,\xi)\in\Gamma_a^{b,\pm}$, $s\geq(\sqrt2+1)/3$ and $y\in\mathbb{R}^{n}$ satisfying $|y-(x\pm s\xi)|\leq a(s+b)/3$.\\
\end{proof}

Now we give an alternative proof of the existence of the wave operators $W_\pm(\tau)$ by using wave packet transform.
\begin{Prop}[D. R. Yafaev \cite{Yafaev}]\label{PropExis}
Suppose that $\mathrm{(A)}$ and $\mathrm{(B)}$ be satisfied.
Then the wave operators
$W_{\pm}(\tau)$
exist for any $\tau\in\mathbb{R}$.
\end{Prop}

\begin{proof}
Substituting $V(t-\tau,x)$ for $V(t,x)$, it suffices to show the case $\tau=0$.
We prove the existence in the case $t\rightarrow+\infty$ only.
Let $\Phi\in\mathcal{S}_{scat}$ and $u_0\in \mathcal{H}$.

First, we show the existence of $W_+(0)u_0$ for $W_\Phi u_0\in C_0^\infty(\mathbb{R}^{2n}\setminus\{\xi=0\})$.
Let $a$ and $R$ be positive constants satisfying
\begin{align}
\label{siki-suppW}
\mathrm{supp}\,W_\Phi u_0\subset\mathbb{R}^{2n}\setminus\Gamma_{a,R}
\end{align}
and $\varphi_0\in\mathcal{S}\setminus\{0\}$ satisfying
\begin{align}
\label{siki-non-orth}
\mathrm{supp}\,\hat{\varphi}_0\subset\left\{\xi\in\mathbb{R}^n\Big|\,\frac{L}2<|\xi|<L\right\}\mbox{ with }0<L\leq \frac a6\quad\mbox{and}\quad |( \Phi,\varphi_0)_{\mathcal{H}}|>0.
\end{align}
By $(\ref{o-m})$ and $(\ref{siki-wpt-inv})$, we have for $t\geq0$
\begin{align}
\label{siki-exis-dual}
(U(0,t)&e^{-itH_0}u_0,\psi)_{\mathcal{H}}\nonumber\\
&=\left( u_0,e^{itH_0}U(t,0)\psi\right)_{\mathcal{H}}\nonumber\\
&=\frac{1}{\left(\varphi_{0},\Phi\right)_{\mathcal{H}}}\left( W_\Phi u_0,W_{\varphi_{0}}[e^{itH_0}U(t,0)\psi]\right)\\
&=\frac{1}{(\varphi_{0},\Phi)_{\mathcal{H}}}\left(W_\Phi u_0,W_{\varphi_{0}}\psi-i\int_0^tW_{\varphi(s)}\left[V(s)U(s,0)\psi\right](x+s\xi,\xi)ds\right).\nonumber
\end{align}
Lemma \ref{no2-term} and $(\ref{siki-suppW})$ show that for $t'\geq t>0$
\begin{align*}
&\left|\left( W_\Phi u_0,-i\int_{t}^{t'}W_{\varphi(s)}\left[V(s)U(s,0)\psi\right](x+s\xi,\xi)ds\right)\right|\\
&\hspace{10mm}\leq\|W_\Phi u_0\| \int_t^{t'}\|W_{\varphi(s)}\left[V(s)U(s,0)\psi\right](x+s\xi,\xi)\|_{L^2(\mathbb{R}^{2n}\setminus\Gamma_{a,R})} ds\\
&\hspace{10mm}\leq \|u_0\|_{\mathcal{H}} \int_t^{t'}C\langle s\rangle^{-\delta}\|U(s,0)\psi\|_{\mathcal{H}} ds\\
&\hspace{10mm}\leq C \langle t\rangle^{1-\delta}\|u_0\|_{\mathcal{H}}\|\psi\|_{\mathcal{H}}.
\end{align*}
The above inequality and $(\ref{siki-exis-dual})$ imply the existence of $W_+(0)u_0$ for $u_0\in W_\Phi^{-1}(C_0^\infty(\mathbb{R}^{2n}\setminus\{\xi=0\}))$.

For $u_0\in \mathcal{H}$, the existence of $W_+(0)u_0$ follows from the fact that $W_\Phi^{-1}(C_0^\infty(\mathbb{R}^{2n}\setminus\{\xi=0\}))$ is dense in $\mathcal{H}$.
Indeed, let $\varepsilon$ be a fixed positive number.
Since $C_0^\infty(\mathbb{R}^{2n}\setminus\{\xi=0\})$ is dense in $L^2(\mathbb{R}^{2n})$, there exists $\omega\in C_0^\infty(\mathbb{R}^{2n}\setminus\{\xi=0\})$ satisfying
$
\|W_\Phi u_0-\omega\|\leq\varepsilon.
$
Putting $\tilde{u_0}=W_\Phi^{-1}\omega$, we have
$
\|U(0,t')e^{-it'H_0}u_0-U(0,t)e^{-itH_0}u_0\|_{\mathcal{H}}
\leq\|U(0,t')e^{-it'H_0}\tilde{u_0}-U(0,t)e^{-itH_0}\tilde{u_0}\|_{\mathcal{H}}
+2\varepsilon
$
for any $t'\geq t>0$.
$(U(0,t)e^{-itH_0}\tilde{u_0})$ is a Cauchy sequence with respect to $t$ as $t\to\infty$ in $\mathcal{H}$, so is $(U(0,t)e^{-itH_0}u_0)$.
\end{proof}
%
%
\section{Characterization of the range of the wave operators}
In this section, we characterize the ranges of the wave operators by the wave packet transform.

\begin{Prop}\label{PropComp}
Suppose that $\mathrm{(A)}$ and $\mathrm{(B)}$ be satisfied.
Then we have
\[
\mathcal{R}(W_{\pm}(\tau))=D_{scat}^{\pm,\Phi}(\tau)
\]
for any $\Phi\in\mathcal{S}_{scat}$.
\end{Prop}

\begin{proof}
It suffices to prove that
$
\mathcal{R}(W_+(0))=D_{scat}^{+,\Phi}(0),
$
since the proposition can be proved in the same way in the other cases.

Let $\Phi\in\mathcal{S}_{scat}$ and  $\varepsilon$ be a fixed positive number.
Until the end of the proof,
we abbreviate $W_+=W_+(0)$,
$D_{scat}^+=D_{scat}^{+,\Phi}(0)$ and $\tilde{D}_{scat}^+=\tilde{D}_{scat}^{+,\Phi}(0)$.

We first prove that
$\mathcal{R}(W_+)\subset D_{scat}^+.
$
Let $f\in\mathcal{R}(W_+)$ and
we fix $g\in W_{\Phi}^{-1}(C_0^\infty(\mathbb{R}^{2n}\setminus\{\xi=0\}))$ satisfying
\begin{align}
\label{siki-d-kinzi}
\|f-W_+g\|_{\mathcal{H}}\leq\varepsilon.
\end{align}
Then there exist positive constants $a$ and $R$ such that
$\chi_{\Gamma_{a,R}}(x,\xi)W_{\Phi}g(x,\xi)=0$ for all $(x,\xi)\in\mathbb{R}^{2n}$.
By $(\ref{siki-psi})$ and the definition of $W_+$, we obtain
\begin{align*}
&\limsup_{t\rightarrow\infty}\|\chi_{\Gamma_{a,R}}W_{\Phi(t)}[U(t,0)W_+g](x+t\xi,\xi)\|\\
&\leq\lim_{t\rightarrow\infty}\left(\|\chi_{\Gamma_{a,R}}W_{\Phi(t)}[e^{-itH_0}g](x+t\xi,\xi)\|
+\|W_{\Phi(t)}[U(t,0)W_+g-e^{-itH_0}g]\|\right)\\
&=\|\chi_{\Gamma_{a,R}}W_{\Phi}g\|+\lim_{t\rightarrow\infty}
\|U(t,0)W_+g-e^{-itH_0}g\|\\
&=0.
\end{align*}
Hence we have $W_+g\in\tilde{D}_{scat}^+$,
which and $(\ref{siki-d-kinzi})$ show $\mathcal{R}(W_+)\subset D_{scat}^+$.

Next we prove
$
\mathcal{R}(W_+)\supset D_{scat}^+.
$
If the inverse wave operator
\begin{align}
\label{siki-inv-wo}
W_{+}^{-1}u_0=\mbox{s-}\hspace{-2mm}\lim_{t\rightarrow+\infty}e^{itH_0}U(t,0)u_0
\end{align}
exists for any $u_0\in D_{scat}^+$,
we obtain $\mathcal{R}(W_+)\supset D_{scat}^+$.
It suffices to prove that $(\ref{siki-inv-wo})$ exists for $u_0\in\tilde{D}_{scat}^{+}$,
since $D_{scat}^{\pm}$ is the closure of $\tilde{D}_{scat}^{\pm}$.

Let $u_0\in\tilde{D}_{scat}^{+}$ and let $a$ and $R$ be positive constants satisfying
\begin{align}
\label{siki-gc}
\lim_{t\rightarrow\infty}\|\chi_{\Gamma_{a,R}}W_{\Phi(t)}[U(t,0)u_0](x+t\xi,\xi)\|=0.
\end{align}
Until the end of the proof,
we abbreviate $\Gamma=\Gamma_{a,R}$ and $\Gamma^c=\mathbb{R}^{2n}\setminus\Gamma$.
Take $\varphi_0\in\mathcal{S}\setminus\{0\}$ satisfying $(\ref{siki-non-orth})$.
By $(\ref{siki-wpt-inv})$, we have for $t'\geq t>0$
\begin{align*}
\left( e^{itH_0}U(t,0)u_0,\psi\right)_{\mathcal{H}}
&=\left( U(t,0)u_0,e^{-itH_0}\psi\right)_{\mathcal{H}}\\
&=\frac{1}{\left(\varphi(t),\Phi(t)\right)_{\mathcal{H}}}\left( W_{\Phi(t)}[U(t,0)u_0],W_{\varphi(t)}[e^{-itH_0}\psi]\right)\\
&=\frac{1}{\left(\varphi_{0},\Phi\right)_{\mathcal{H}}}\left( \chi_{\Gamma}(x-t\xi,\xi)W_{\Phi(t)}[U(t,0)u_0],W_{\varphi(t)}[e^{-itH_0}\psi]\right)\\
&+\frac{1}{\left(\varphi_{0},\Phi\right)_{\mathcal{H}}}\left( \chi_{\Gamma^c}(x-t\xi,\xi)W_{\Phi(t)}[U(t,0)u_0],W_{\varphi(t)}[e^{-itH_0}\psi]\right)
\end{align*}
and
\begin{align*}
( &e^{it'H_0}U(t',0)u_0,\psi)_{\mathcal{H}}\\
&=\left( U(t,0)u_0,U(t,t')e^{-it'H_0}\psi\right)_{\mathcal{H}}\nonumber\\
&=\frac{1}{\left(\varphi_{0},\Phi\right)_{\mathcal{H}}}\left( (\chi_{\Gamma}+\chi_{\Gamma^c})(x-t\xi,\xi)W_{\Phi(t)}[U(t,0)u_0],W_{\varphi(t)}[U(t,t')e^{-it'H_0}\psi]\right).
\end{align*}
Taking the difference between the above equalities, we have
\begin{align}
\label{siki-decom}
&\left(\varphi_{0},\Phi\right)_{\mathcal{H}}\Big(e^{itH_0}U(t,0)u_0-e^{it'H_0}U(t',0)u_0,\psi\Big)_{\mathcal{H}}\nonumber\\
&=
\left( \chi_{\Gamma}(x-t\xi,\xi)W_{\Phi(t)}[U(t,0)u_0],W_{\varphi(t)}[e^{-itH_0}\psi-U(t,t')e^{-it'H_0}\psi]\right)\\
&+
\left(W_{\Phi(t)}[U(t,0)u_0],\chi_{\Gamma^c}(x-t\xi,\xi)\left(W_{\varphi(t)}[e^{-itH_0}\psi-U(t,t')e^{-it'H_0}\psi]\right)\right).\nonumber
\end{align}
Using $(\ref{siki-gc})$, we obtain
\begin{align}
\label{siki-G}
&\sup_{\|\psi\|_{\mathcal{H}}=1}|(\mbox{the } \mbox{first term of the right hand side in } (\ref{siki-decom}))|\nonumber\\
&\leq\sup_{\|\psi\|_{\mathcal{H}}=1}\|\chi_{\Gamma}(x-t\xi,\xi)W_{\Phi(t)}[U(t,0)u_0]\|\|W_{\varphi(t)}[e^{-itH_0}\psi-U(t,t')e^{-it'H_0}\psi]\|\nonumber\\
&\leq2\|\varphi_0\|_{\mathcal{H}}\|\chi_{\Gamma}W_{\Phi(t)}[U(t,0)u_0](x+t\xi,\xi)\|.
\end{align}
We have by $(\ref{siki-psi})$ and $(\ref{siki-psi2})$
\begin{align*}
W_{\varphi(t)}[&e^{-itH_0}\psi-U(t,t')e^{-it'H_0}\psi](x,\xi)\\
&=-i\int_t^{t'} e^{-i\frac{1}{2}(t-s) |\xi|^2}
W_{\varphi(s)}\left[V(s)U(s,t')e^{-it'H_0}\psi\right](x-(t-s)\xi,\xi)ds,
\end{align*}
which shows by Lemma \ref{no2-term}
\begin{align}
\label{siki-Gc}
&|(\mbox{the } \mbox{second term of the right hand side in } (\ref{siki-decom}))|\nonumber\\
=&\Bigg|\Bigg( W_{\Phi(t)}[U(t,0)u_0](x+t\xi,\xi),\chi_{\Gamma^c}(x,\xi)\int_t^{t'} e^{-i\frac{1}{2}(t-s) |\xi|^2}\nonumber\\
&\hspace{47mm}\times W_{\varphi(s)}\left[V(s)U(s,t')e^{-it'H_0}\psi\right](x+s\xi,\xi)ds\Bigg)\Bigg|\nonumber\\
\leq& C\|u_0\|_{\mathcal{H}}\int_t^{t'}\left\|W_{\varphi(s)}\left[V(s)U(s,t')e^{-it'H_0}\psi\right](x+s\xi,\xi)\right\|_{L^2(\Gamma^c)}ds\\
\leq& C\|u_0\|_{\mathcal{H}}\langle t\rangle^{-\delta+1}\|\psi\|_{\mathcal{H}}.\nonumber
\end{align}
$(\ref{siki-inv-wo})$ follows from $(\ref{siki-G})$ and $(\ref{siki-Gc})$.
\end{proof}
Theorem \ref{main-theorem} is obtained by Proposition \ref{PropExis} and \ref{PropComp}.

%
\section{Proof of Theorem \ref{indp-theorem} by our characterization}
In this section, we give an alternative proof of Theorem \ref{indp-theorem} by using our characterization.

\begin{proof}
We shall only prove for $\tau=0$ and for the case that $t\rightarrow+\infty$.
We fix $\Phi\in\mathcal{S}_{scat}$.
We use the same notations $D_{scat}^+=D_{scat}^{+,\Phi}(0)$ and $\tilde{D}_{scat}^+=\tilde{D}_{scat}^{+,\Phi}(0)$
as in the proof of Proposition \ref{PropComp}.

Firstly, we prove $\mathcal{H}_{p}(H)^{\perp}\supset D_{scat}^+$.
For $u_0\in \tilde{D}_{scat}^{+}$, we have
\begin{align}
\label{siki-aa}
\lim_{t\rightarrow\infty}\|\chi_{\Gamma_{a,R}}W_{\Phi(t)}[e^{-itH}u_0](x+t\xi,\xi)\|=0
\end{align}
for some positive constants $a$ and $R$.
On the other hand, for $\omega\in\mathcal{H}_{p}(H)$, 
we have $\omega=\sum_{n=1}^\infty a_n\omega_n$ where $a_n\in\mathbb{C}$ with $\sum_{n=1}^{\infty}|a_n|^2<\infty$
and $\omega_n$ is the normalized eigenfunction of $H$ corresponding to the eigenvalue $\lambda_n$.
Then we see that
$e^{-itH}\omega=\sum_{n=1}^\infty a_ne^{-it\lambda_n}\omega_n$ for any $t\in\mathbb{R}$.
Taking $\varphi_0\in \mathcal{S}$ satisfying
$(\ref{siki-non-orth})$, we get for $t\geq0$
\begin{align}
(e^{-itH}&u_0,e^{-it\lambda_n}\omega_n)_{\mathcal{H}}\nonumber\\
\label{siki-ip}
&=\frac{1}{\left(\varphi_{0},\Phi\right)_{\mathcal{H}}}\left(\chi_{\Gamma_{a,R}}W_{\Phi(t)} [e^{-itH}u_0](x+t\xi,\xi),e^{-it\lambda_n}W_{\varphi(t)}\omega_n(x+t\xi,\xi)\right)\\
&+\frac{1}{\left(\varphi_{0},\Phi\right)_{\mathcal{H}}}\left(W_{\Phi(t)} [e^{-itH}u_0](x+t\xi,\xi),
\chi_{\Gamma_{a,R}^c}e^{-it\lambda_n}W_{\varphi(t)}\omega_n(x+t\xi,\xi)\right).\nonumber
\end{align}
From $(\ref{siki-aa})$, the first term of $(\ref{siki-ip})$ is estimated by
\begin{align}
\label{siki-ad}
\lim_{t\to+\infty}\Big|&\left(\chi_{\Gamma_{a,R}}W_{\Phi(t)} [e^{-itH}u_0](x+t\xi,\xi),e^{-it\lambda_n}W_{\varphi(t)}\omega_n(x+t\xi,\xi)\right)\Big|\nonumber\\
&\leq C\|\omega_n\|_{\mathcal{H}}\lim_{t\to+\infty}\|\chi_{\Gamma_{a,R}}W_{\Phi(t)} [e^{-itH}u_0](x+t\xi,\xi)\|=0.
\end{align}
The second term of $(\ref{siki-ip})$ is estimated by
\begin{align}
\label{siki-ac}
&\lim_{t\to+\infty}\Big|\left(W_{\Phi(t)} [e^{-itH}u_0](x+t\xi,\xi),
\chi_{\Gamma_{a,R}^c}e^{-it\lambda_n}W_{\varphi(t)}\omega_n(x+t\xi,\xi)\right)\Big|\nonumber\\
&\leq \|u_0\|_{\mathcal{H}}\lim_{t\to+\infty}\|\chi_{\Gamma_{a,R}^c}W_{\varphi(t)}\omega_n(x+t\xi,\xi)\|\\
&\leq \|u_0\|_{\mathcal{H}}\lim_{t\to+\infty}\left\|\chi_{\left\{(x,\xi)\big||x+t\xi|\geq at/2\right\}}(x,\xi)\int_{\mathbb{R}^{n}}\varphi(t,y-(x+t\xi))\omega_n(y)e^{-iy\xi}dy\right\|\nonumber\\
&\leq \|u_0\|_{\mathcal{H}}\lim_{t\to+\infty}\left(\|F(|x|>at/3)\varphi(t)\|_{\mathcal{H}}\|\omega_n\|_{\mathcal{H}}
+\|\varphi_0\|_{\mathcal{H}}\|F(|x|>at/6)\omega_n\|_{\mathcal{H}}\right)\nonumber\\
&=0.\nonumber
\end{align}
$(\ref{siki-ac})$ follows from Lemma \ref{kuroda-lem}.
$(\ref{siki-ad})$ and $(\ref{siki-ac})$ implies $(u_0,\omega)_{\mathcal{H}}=0$.
Thus we obtain $\mathcal{H}_{p}(H)^{\perp}\supset D_{scat}^{+}$.

Secondly, we prove $\mathcal{H}_{p}(H)^{\perp}\subset D_{scat}^+$.
We put $\tilde{\mathcal{H}}_{p}(H)^{\perp}=\{E_H((a',b'))f|f\in\mathcal{H}_{p}(H)^{\perp}\mbox{ and }0<a'<b'\}$,
then $\tilde{\mathcal{H}}_{p}(H)^{\perp}$ is dense in $\mathcal{H}_{p}(H)^{\perp}$,
where $E_H(\mathcal{B})$ is a spectral measure of $H$ for a Borel set $\mathcal{B}$.
Thus it suffices to prove $\tilde{\mathcal{H}}_{p}(H)^{\perp}\subset \tilde{D}_{scat}^+$.
Let $f\in\tilde{\mathcal{H}}_{p}(H)^{\perp}$
and be a positive constant $d$ and $\phi\in C_0^\infty([0,\infty))$ satisfying that $\phi(H)f=f$ and that $\phi\equiv0$ on $[0,d^2/2]$.
Since $\mbox{w-}\hspace{-1mm}\lim_{t\to\infty}e^{-itH}f=0$ in $\mathcal{H}$ and $(\phi(H)-\phi(H_0))$ is a compact operator on $\mathcal{H}$,
we have
\begin{align}
\label{siki-w-c}
\lim_{t\to\infty}\left\|\left(\phi(H)-\phi(H_0)\right)e^{-itH}f\right\|_{\mathcal{H}}=0.
\end{align}
Let $\Xi_d=\{(x,\xi)\in\mathbb{R}^{2n}|\,|\xi|\leq d\}$ and $\Xi_d^r=\{(x,\xi)\in\mathbb{R}^{2n}|\,|\xi|> d\mbox{ and }|x|\leq r\}$.
Then there exists a sequence $(t_N)$ tending to $\infty$ as $N\to\infty$ such that
\begin{align}
\label{siki-Gd}
\lim_{N\to\infty}\|\chi_{\Xi_d} W_{\Phi}[e^{-it_NH}f]\|=0
\end{align}
and
\begin{align}
\label{siki-Gr}
\lim_{N\to\infty}\|\chi_{\Xi_d^r} W_{\Phi}[e^{-it_NH}f]\|=0
\end{align}
for any positive constant $r$.
Indeed,
since $\phi(|\xi|^2/2)\chi_{\Xi_d}(y,\xi)\equiv0$ for any $(y,\xi)\in\mathbb{R}^{2n}$, we obtain
\begin{align*}
\phi(H_0)W_{\Phi}^{-1}[\chi_{\Xi_d}\Psi]
&=\mathcal{F}^{-1}\left[\phi\left(\frac{|\xi|^2}{2}\right) \mathcal{F}W_{\Phi}^{-1}\chi_{\Xi_d}\Psi\right]\nonumber\\
&=\mathcal{F}^{-1}\left[\phi\left(\frac{|\xi|^2}{2}\right)\left[\int\Phi(x-y)\chi_{\Xi_d}(y,\xi)\Psi(y,\xi)dy\right]\right]=0.
\end{align*}
Thus we have for $t\geq0$
\begin{align}
\label{siki-Culd}
(\chi_{\Xi_d} &W_{\Phi}[e^{-itH}f](x,\xi),\Psi)\nonumber\\
&=\left(\phi(H)e^{-itH}f,W_{\Phi}^{-1}\chi_{\Xi_d}\Psi\right)_{\mathcal{H}}\\
&=\left(e^{-itH}f,\phi(H_0)W_{\Phi}^{-1}\chi_{\Xi_d}\Psi\right)_{\mathcal{H}}+\left((\phi(H)-\phi(H_0))e^{-itH}f,W_{\Phi}^{-1}\chi_{\Xi_d}\Psi\right)_{\mathcal{H}}\nonumber\\
&=\left((\phi(H)-\phi(H_0))e^{-itH}f,W_{\Phi}^{-1}\chi_{\Xi_d}\Psi\right)_{\mathcal{H}}.\nonumber
\end{align}
$(\ref{siki-w-c})$, $(\ref{siki-Culd})$ and the fact that $\|W_{\Phi}^{-1}\chi_{\Xi_d}\Psi\|_{\mathcal{H}}\leq\|\Psi\|$ yield $(\ref{siki-Gd})$.
Take $\Phi'\in C_0^\infty(\mathbb{R}^{n})$ with $|(\Phi,\Phi')_{\mathcal{H}}|>0$.
For any positive constant $r$, there exists a positive constant $r'$ satisfying 
$\chi_{\Xi_d^r} W_{\Phi'}[e^{-itH}f](x,\xi)=\chi_{\Xi_d^r} W_{\Phi'}[F(|x|<r')e^{-itH}f](x,\xi)$ for any $t\geq0$.
By the RAGE theorem (\cite{AG}, \cite{Ruelle}), there exists a sequence $(t_N)$ tending to $\infty$ as $N\to\infty$ such that
$
\lim_{N\to\infty}\|F(|x|<r')e^{-it_NH}f\|_{\mathcal{H}}=0.
$
Hence we get $(\ref{siki-Gr})$.

For any positive number $\varepsilon$, we take a positive constant $r$ sufficiently large so that $C\|f\|_{\mathcal{H}}\langle r\rangle^{1-\delta}\leq \varepsilon/4$.
Since $(\chi_{\Xi_d}+\chi_{\Xi_d^r}+\chi_{\Gamma_{d}^{r,-}}+\chi_{\Gamma_{d}^{r,-}})(x,\xi)=1$ for almost all $(x,\xi)\in\mathbb{R}^{2n}$, we have for $t,t'\geq t_N$
\begin{align*}
&\big(e^{it'H_0}e^{-it'H}f-e^{itH_0}e^{-itH}f,\psi\big)_{\mathcal{H}}\\
&=\left(W_\Phi[e^{-it_NH}f], W_\Phi \left[\left(e^{-i(t_N-t')H}e^{-it'H_0}-e^{-i(t_N-t)H}e^{-itH_0}\right)\psi\right]\right)\\
&=\big(\big(\chi_{\Xi_d}+\chi_{\Xi_d^r}+\chi_{\Gamma_{d}^{r,-}}\big)W_\Phi[e^{-it_NH}f],
W_\Phi \big[\big(e^{-i(t_N-t')H}e^{-it'H_0}\\
&\hspace{80mm}-e^{-i(t_N-t)H}e^{-itH_0}\big)\psi\big]\big)\\
&+\left(W_\Phi[e^{-it_NH}f], \chi_{\Gamma_{d}^{r,+}}\left(I(t'-t_N;x,\xi)-I(t-t_N;x,\xi)\right)\right),
\end{align*}
where $I(t;x,\xi)=-i\int_0^{t}e^{\frac{i}{2}(s-t) |\xi|^2}W_{\Phi(s)}[Ve^{-i(s-t)H}e^{-i(t+t_N)H_0}\psi](x+s\xi,\xi)ds$.\\
By the equality $(\ref{siki-psi-V})$ with $t=0$ and $t_0=-t_N$, we have
\begin{align}
\label{siki-r--term}
\chi_{\Gamma_{d}^{r,-}}&W_{\Phi}[e^{-it_NH}f](x,\xi)\nonumber\\
&=\chi_{\Gamma_{d}^{r,-}}W_{\Phi}[U(0,-t_N)f](x,\xi)\\
&=\chi_{\Gamma_{d}^{r,-}}W_{\Phi(-t_N)}f(x-t_N\xi,\xi)\nonumber\\
&\hspace{10mm}-i\int_0^{t_N}e^{\frac{i}{2}(-s+t_N)|\xi|^2}\chi_{\Gamma_{d}^{r,-}}W_{\Phi(-s)}[Ve^{i(s+t_N)H}f](x-s\xi,\xi)ds.\nonumber
\end{align}
Lemma \ref{no2-term2} shows that
\begin{align}
\label{siki-G+}
\|\chi_{\Gamma_{d}^{r,+}}I(t-t_N;x,\xi)\|
\leq C' \|\psi\|_{\mathcal{H}}\int_0^{t-t_N}\langle s+r\rangle^{-\delta} ds
\leq C \|\psi\|_{\mathcal{H}}\langle r\rangle^{1-\delta}
\end{align}
and
\begin{align}
\label{siki-G-i}
\int_0^{t_N}\|\chi_{\Gamma_{d}^{r,-}}W_{\Phi(-s)}[Ve^{i(s+t_N)H}f](x-s\xi,\xi)\|ds
\leq C\|f\|_{\mathcal{H}}\langle r\rangle^{1-\delta},
\end{align}
where $C$ and $C'$ is a positive constant and independent of $r$ and $N$.
Taking $\varphi_0\in\mathcal{S}$ satisfying $(\ref{siki-non-orth})$ with $a=d$,
we have by Lemma \ref{kuroda-lem}
\begin{align}
\label{siki-G-2}
&\lim_{N\to\infty}\|\chi_{\Gamma_{d}^{r,-}}W_{\Phi(-t_N)}f(x-t_N\xi,\xi)\|\nonumber\\
&\leq\lim_{N\to\infty}\frac{1}{\left|(\varphi_{0},\Phi)_{\mathcal{H}}\right|}\|\chi_{\Gamma_{d}^{r,-}}W_{\varphi(-t_N)}f(x-t_N\xi,\xi)\|\\
&\leq\lim_{N\to\infty}\frac{1}{\left|(\varphi_{0},\Phi)_{\mathcal{H}}\right|}\left(\|F(|x|>dt_N/3)\varphi(-t_N)\|_{\mathcal{H}}\|f\|_{\mathcal{H}}
+\|\varphi_0\|_{\mathcal{H}}\|F(|x|>dt_N/6)f\|_{\mathcal{H}}\right)\nonumber\\
&=0\nonumber.
\end{align}
By $(\ref{siki-r--term})$, $(\ref{siki-G-i})$ and $(\ref{siki-G-2})$,
we obtain
\begin{align}
\label{siki-G-}
\limsup_{N\to\infty}\|\chi_{\Gamma_{d}^{r,-}}W_{\Phi}[e^{-it_NH}f]\|\leq\varepsilon.
\end{align}
Thus $(\ref{siki-Gd})$, $(\ref{siki-Gr})$, $(\ref{siki-G+})$ and $(\ref{siki-G-})$ imply that
$
\limsup_{t,t'\to\infty}\|e^{it'H_0}e^{-it'H}f-e^{itH_0}e^{-itH}f\|_{\mathcal{H}}\leq\varepsilon.
$
Hence there exists $\Omega\in \mathcal{H}$ such that
\begin{align}
\label{siki-dens}
\lim_{t\to\infty}\|\chi_{\Gamma_{a,R}}W_\Phi[e^{itH_0}e^{-itH}f-\Omega]\|
\leq\lim_{t\to\infty}\|e^{itH_0}e^{-itH}f-\Omega\|_{\mathcal{H}}=0.
\end{align}
Since $W_{\Phi}^{-1}(C_0^\infty(\mathbb{R}^{2n}\setminus\{\xi=0\}))$ is dense in $\mathcal{H}$, 
for any positive number $\varepsilon'$ there exist $\Omega'\in W_{\Phi}^{-1}(C_0^\infty(\mathbb{R}^{2n}\setminus\{\xi=0\}))$ and positive constants $a$ and $R$ satisfying
\begin{align}
\label{siki-dens1}
\|\Omega-\Omega'\|_{\mathcal{H}}\leq\varepsilon' \mbox{ and } \chi_{\Gamma_{a,R}}W_{\Phi}\Omega'\equiv0.
\end{align}

$(\ref{siki-dens})$ and $(\ref{siki-dens1})$ yield $\mathcal{H}_{p}(H)^{\perp}\subset D_{scat}^+$,
which and the first part complete the proof.
\end{proof}

\begin{rem}
Thus substituting $(\ref{siki-Gd})$ and $(\ref{siki-Gr})$ for $(\ref{KYtype-scat})$, we have $\mathcal{R}(W_{\pm}(\tau))=\mathcal{Y}_{scat}^{\pm}(\tau)$.
\end{rem}


\end{document}